\documentclass[reqno,12pt,a4paper]{amsart}

\usepackage{color,amsmath,amssymb,latexsym,amsfonts,setspace}
\usepackage[parfill]{parskip}

\normalsize \evensidemargin=0pt \oddsidemargin=0pt \hoffset=6pt
 \def\lms{\mathop{\overline{\lim}}\limits}

\def\lmi{\mathop{\underline{\lim}}\limits}

\def\v{\varepsilon}

\def\d{\delta}
\def\a{\alpha}

\def\l{\lambda}

\def\x{\xi}

\def\C{\mathbb C}
\def\R{\mathbb R}

\def\S{Section }

\def\r{\rho}
\def\x{\xi}

\theoremstyle{plain}
\newtheorem{theorem}{Theorem}[section]
\newtheorem{lemma}{Lemma}[section]

\newtheorem{corollary}{Corollary}[section]

\theoremstyle{definition}
\newtheorem{example}{Example}[section]

\numberwithin{equation}{section}

\begin{document}

%

\title{Large Deviations for processes on half-line}
\author{F.C. Klebaner}
\address{School of Mathematical Sciences, Monash University, Clayton,  VIC 3800, Australia.}
\email{fima.klebaner@monash.edu}
\author{A.V. Logachov}
\address{Novosibirsk State University, Novosibirsk, 2 Pirogova Str.,
630090, Russia.} \email{omboldovskaya@mail.ru}
\author{A.A. Mogulskii}
\address{Sobolev Institute of Mathematics of the Siberian Branch of the RAS, Novosibirsk, 4 Acad. Koptyug avenue, 630090, Russia.} \email{mogul@math.nsc.ru}

\begin{abstract} We consider a sequence of processes
$X_n(t)$ defined on half-line \\$0\leq t < \infty$.
We give sufficient conditions for Large Deviation Principle (LDP) to hold in the space of continuous functions with metric
$
  \r_\kappa(f,g)=\sup\limits_{t\ge 0}\frac{|f(t)-g(t)|}{1+t^{1+\kappa}},$ $\kappa\geq 0.$
LDP is established for Random Walks, Diffusions, and CEV  model of ruin, all defined on the half-line. LDP in this space is ``more precise" than that with the usual metric of uniform convergence on compacts.
\end{abstract}

\maketitle



\section{Introduction}

In this work we   derive sufficient conditions for a sequence $\{X_n\}_{n=1}^{\infty}$ of stochastic processes $X_n(t)$;
$0\leq t < \infty$,  to satisfy the Large Deviation Principle (LDP)
in the space of continuous functions on $[0,\infty)$, which we denote by $\C$.

In the recent literature
\cite{Puh1}, \cite{Dem}, \cite{Kur} the space $\C$ is considered with the metric

\begin{equation}{\label{0.12}}
  \r^{(P)}(f,g):=\sum_{k=1}^\infty 2^{-k}\min\{\sup_{0\le t\le k}|f(t)-g(t)|, 1\}.
\end{equation}
  \cite{Puh2} (Theorem 2.6) gives sufficient conditions for
$X_n$ to satisfy LDP in the space $(\C,\r^{(P)})$.

As noted in  \cite{Puh2}, convergence $f_n\rightarrow f$ in metric $\r^{(P)}$
is equivalent to convergence in $\C[0,T]$ with uniform metric for any $T\geq 0$. A considerable drawback of   metric  $\r^{(P)}$ is that it  is `` not sensitive" \;  to behaviour of functions as $t\rightarrow \infty$.

We consider the space  $\C$ with   metric
$$
  \r(f,g)=\r_\kappa(f,g):=\sup_{t\ge 0}\frac{|f(t)-g(t)|}{1+t^{1+\kappa}},
$$
for a fixed $\kappa\ge 0$. It is obvious that
$(\C,\r)$  is a complete separable metric (Polish) space.

As we shall see in  \S2, the LDP in the space $(\C,\r)$
is   ``more precise"\; than the LDP in $(\C,\r^{(P)})$.

In this work we treat continuous  processes   on   infinite interval. As we envisage, a treatment of discontinuous processes on   infinite interval will need
essentially different to $\r$ metric,  (see \cite{PechDob}, for the LDP for Compound Poisson processes on  infinite interval).
Note that in \cite{Str}, Theorem 1.3.27, LDP for Wiener process in space $(\C,\r_\kappa)$ when $\kappa=0$ is given, while in \cite{Mak} the Law of Iterated Logarithm is proved for   Wiener process in this space.

The paper is organised as follows.
Sufficient conditions for LDP in the space $(\C,\r)$ are given in \S 2, Theorem 2.1. We also   compare Theorem 2.1 and Theorem 2.6 of \cite{Puh2}, and show that Theorem 2.1 is more precise.
Next we apply Theorem 2.1 to different kind of processes,   such as Random Walks and Diffusions on half line. Only Random Walks case is given here, and the reader referred to the Arxiv version for other examples.

\section{Main Result}

To formulate the main result we give a number of definitions.  For any $T\in (0,\infty)$
 denote by $\C[0,T]$ the metric space of real continuous functions $f=f(t);~0\le t\le T$, with metric
 $$
 \r_T(f,g):=\sup_{0\le t \le T}\frac{|f(t)-g(t)|}{1+t^{1+\kappa}},
 $$
 where $\kappa\ge 0$ is fixed.

We say that in space
 $\C[0,T]$ there is a  (good) rate function
 $$
 I_0^T=I_0^T(f):~\C[0,T]\to [0,\infty],
 $$
 if: $(i)$ it is lower semi-continuous: {\it for any $f\in \C[0,T]$}
 \begin{equation}{\label{0.1}}
   \lmi_{f_n\to f}I_0^T(f_n)\ge I_0^T(f);
\end{equation}
 $(ii)$  {\it for any $r\ge 0$  the set
 $$
   B_{T,r}:=\{f\in \C[0,T]:~I_0^T(f)\le r\}
 $$
 is a compact in $\C[0,T]$. }

For a non-empty set $B\subset \C[0,T]$  let
$$
  I_0^T(B):=\inf_{f\in B}I_0^T(f),~~~I_0^T(\emptyset):=\infty.
    $$
 $(f)_{T,\v}$ and  $(B)_{T,\v}$
denote $\v$-neighbourhood in metric $\r_T$ in space $\C[0,T]$
of $f\in \C[0,T]$ and measurable set
$B\subset \C[0,T]$ respectively. The interior and the closure of a measurable set $B\subset \C[0,T]$ is denoted by $(B)_T$ and $[B]_T$  respectively.

Note that lower semi-continuity (\ref{0.1})
can be written as: {\it for any
$f\in \C[0,T]$}
\begin{equation}{\label{0.2}}
   \lim_{\v\to 0}I_0^T((f)_{T,\v})=I_0^T(f).
\end{equation}
It is obvious that (\ref{0.1}) and  (\ref{0.2})  are equivalent.


For a function $f\in \C$, $f^{(T)}$ denotes its projection on $\C[0,T]$,
$$
f^{(T)}=f^{(T)}(t):=f(t);~0\le t\le T.
$$

Denote by $\C_0\subset \C$ -- the class of functions $f\in \C$, such that $f(0)=0$,
$\lim\limits_{t\to \infty}\frac{f(t)}{1+t^{1+\kappa}}=0$.\\

Let now   $X_n(t);~t\in [0,\infty)$, be a sequence of processes in space $\C_0$.
We  assume  the following conditions.

{\bf I}. {\it For any $T\in (0,\infty)$
processes $X_n^{(T)}$  satisfy LDP in space
$\C[0,T]$  with good rate function $I_0^T$,
i.e.  for any measurable set
$B\subset \C[0,T]$
        $$
        \lms_{n\to \infty}\frac{1}{n}\ln {\bf P}(X^{(T)}_n\in B)\le -I_0^T([B]_T),
        $$
         $$
        \lmi_{n\to \infty}\frac{1}{n}
        \ln {\bf P}(X^{(T)}_n\in B)\ge -I_0^T((B)_T).
        $$

Moreover, for any $f\in \C[0,T]$
there is $g=g_f\in \C_0$, such that $g^{(T)}=f$, and for any
$U\ge T$ it holds
\begin{equation}{\label{0.3}}
      I_0^U(g^{(U)})=I_0^{T}(f).
        \end{equation}
}

Condition (\ref{0.3}) means that one can extend any
$f\in \C[0,T]$ for  $t>T$
such that the rate function will stay the same. It is natural to call the function
$g=g_f$ the most likely extension of $f$ beyond $[0,T]$.

{\bf II}. {\it For any $r\ge 0$
$$
  \lim_{T\to \infty}\sup_{f\in B_r^+} \sup_{t\ge T}\frac{|f(t)|}{1+t^{1+\kappa}}=0,
$$
where
$$
 B_r^+:=\{f\in \C:~\lms_{T\to \infty}I_0^T(f^{(T)})\le r\}.
$$
}

{\bf III}. {\it For any $N<\infty$ and $\v>0$
there is $T=T_{N,\v}<\infty$ such that
$$
\lms_{n\to \infty}\frac{1}{n} \ln
{\bf P}(\sup_{t\ge T}\frac{|X_n(t)|}{1+t^{1+\kappa}}>\v)\le -N.
$$
}

\begin{theorem} \label{thm}
Assume conditions {\bf I, II and III.}  Then
for any
$f\in \C$ there exists
\begin{equation}{\label{N3.1}}
\lim_{T\to \infty}I_0^T(f^{(T)})=:I(f),
\end{equation}
 and  it is a good rate function in the space $(\C,~\r)$.
The sequence  $X_n$ satisfies LDP
 in this space   with rate function $I(f)$, i.e.
 for any measurable $B\subset \C$
         \begin{equation}\label{0.10}
        \lms_{n\to \infty}\frac{1}{n}\ln {\bf P}(X_n\in B)\le -I([B]),
        \end{equation}
         \begin{equation}{\label{0.11}}
        \lmi_{n\to \infty}\frac{1}{n}
        \ln {\bf P}(X_n\in B)\ge -I((B)),
        \end{equation}
where $[B]$, $(B)$ is the closure and the interior of $B$, respectively, and $$
  I(B):=\inf_{f\in B}I(f),
$$
with $I(\emptyset)=\infty$.
 \end{theorem}

Note that if for a  set $B$,
$I([B])=I((B))(=I(B))$,
then  inequalities (\ref{0.10}),  (\ref{0.11})
can be replaced by equality
$$
  \lim_{n\to \infty}\frac{1}{n}\ln {\bf P}(X_n\in B)=-I(B).
$$
Hence the difference
$$
 D ( B ) :=  I((B))-I([B])\geq 0
$$
describes precision of LDP: the smaller the difference the more precise is the theorem.
Theorem 2.6 in \cite{Puh2} gives sufficient conditions    for a sequence
$X_n$ to satisfy LDP in the space $(\C,\r^{(P)})$.

We  compare our Theorem \ref{thm}
and Theorem 2.6 in \cite{Puh2}. It follows that the rate functions in both theorems are the same. This is because projections
$X^{(T)}_n$
on $[0,T]$
satisfy LDP in the space $\C[0,T]$ with uniform metric and rate function $I_0^T(f)$  common for both theorems.
Therefore we can compare these theorems by comparing differences
$$
 D ( B ) := I((B))-I([B])~~~\mbox{and}~~~D ^{(P)}( B ) :=I((B)^{(P)})-I([B]^{(P)}),
$$
where $[B]^{(P)}$, $(B)^{(P)}$ are the closure and the interior of $B$
in metric $\r^{(P)}$, respectively.

As noted earlier, (see also \cite{Puh2}), $\r^{(P)}(f_n,f)\to 0$ as
  $n\to \infty$,
is equivalent to
$\r_T(f_n,f)\to 0$  for any $T>0$.
Therefore  $\r(f_n,f)\to 0$ implies
  $\r^{(P)}(f_n,f)\to 0$. It is  easy to see that the opposite is not true.
Thus
$$
  [B]\subset [B]^{(P)},~~~(B)^{(P)}\subset (B),
$$
therefore
$$
  I([B]^{(P)})\le I([B]),~~~I((B))\le I((B)^{(P)}),
$$
so that we have always $D ( B ) \leq D ^{(P)}( B )$.
Below we give an example  of  $B$  satisfying simultaneously
$$
  I([B])=I((B))\in (0,\infty),~~~I([B]^{(P)})=0.
$$
Hence Theorem 2.1 allows to give  ``precise"
logarithmic asymptotic for
${\bf P}(X_n\in B)$, when Theorem 2.6 in \cite{Puh2} does not.
We conclude that {\it LDP in the space $(\C,\r)$
is more precise than in the space $(\C,\r^{(P)})$}.

\begin{example} Consider Wiener process
$w=w(t)$ on $[0,\infty)$.
Denote
$$
w_n=w_n(t):=\frac{1}{\sqrt{n}}w(t),~~~~t\ge 0.
$$
Since conditions ${\bf I}$---${\bf III}$
are easily checked, then LDP follows from Theorem 2.1 with rate function
$$
  I(f)=  \left\{
           \begin{array}{lcl}
                              \frac{1}{2}\int_0^\infty (f'(t))^2 dt,~~~\mbox{if}~~~f(0)=0, ~~~f \mbox{ is absolutely continuous},\\
                              \infty~~~~\mbox{otherwise}.\\
                              \end{array}
                              \right.
$$

$$
  B=\overline{(f_0)_1}:=\left\{g\in \C:~\sup_{t\ge 0}\frac{|g(t)|}{1+t}\ge 1\right\},~~~f_0=f_0(t)\equiv 0.
$$
Since it is a complement to an open set $(f_0)_1$,
it is closed in $(\C,\r)$, and therefore
$$
  I([B])=I(B)=\inf_{g\in B}I(g).
$$
By Cauchy-Bunyakovski inequality
$$
1\leq \sup_{t\ge 0}\frac{|g(t)|}{1+t}=\sup_{t\ge 0}\frac{|\int_0^t g'(s)ds|}{1+t}\leq
\sup_{t\ge 0}\biggl|\int_0^t (g'(s))^2ds\biggl|^{1/2}\sup_{t\ge 0}\frac{t^{1/2}}{1+t}=\frac{\sqrt{2I(g)}}{2}.
$$
This gives that $I(g)\geq 2$  for all $g \in B$.

Take $f(t)=2tI(0\le t\le 1)+2I(t\ge 1)$.
 It is easy to see that $f\in B$ and $I(f) = 2$. 
Therefore  $I([B]) = I(B) = 2$.

Taking   $f_n(t)= (2+1/n)tI(0\le t\le 1)+2+1/n I(t\ge 1)$, we can see that
$f_n\in (B)$
 and
  $I([B])=I(B)=I((B))=2$.   Hence,
$$
\lim_{n\to \infty}\frac{1}{n}\ln {\bf P}(w_n\in B)=-2.
$$

Consider now $[B]^{(P)}$, the closure of $B$  in metric
$\r^{(P)}$. By taking $g_n(t)=\frac{~t^2}{n}$
it is easy to see that $g_n \in B$ for all    $n$ and
$\lim\limits_{n\to \infty}\r^{(P)}(g_n,f_0)=0$. Therefore, $f_0 \in [B]^{(P)}$.
Therefore $I([B]^{(P)})=0$, and
the upper bound in Theorem~3.4
for the set $B$ is trivial, which does not allow to find logarithmic asymptotic of the required probability.
\end{example}

\section{Proof of   Theorem \ref{thm}}

For  $\v>0$
denote by $(f)_\v$ and $(B)_\v$ the $\v$-neighborhood of $f\in \C$, and set $B\subset \C$,
respectively.

The proof of the Theorem 2.1 consists of three steps.  The first step   proves that $I(f)$ is a good rate function in Lemma \ref{Lemma2.1}.   The second step proves the local LDP for $X_n$ in $(\C_0,\r)$  in Lemma \ref{Lemma2.2} .   The third step proves  a weaker form of exponential tightness for $X_n$ in Lemma \ref{Lemma2.3}.

The
upper bound is obtained  by Lemmas \ref{Lemma2.2} and \ref{Lemma2.3} , for any measurable set $B\subset \C_0$  and $\v>0$
it holds (see e.g. \cite{SMG1}, Theorem 3.1)
$$
  \lms_{n\to \infty}\frac{1}{n}\ln {\bf P}(X_n\in B) \le - I((B)_\v).
$$
As it is known (see e.g. \cite{SMG1}, Lemma 2.1),
that a good rate function $I(f)$  satisfies
$$
  \lim_{\v\to \infty}I((B)_\v)=I([B]),
$$
the upper bound (\ref{0.10}) is proved.
Lower bound  (\ref{0.11})
follows from (\ref{0.14}) of Lemma  \ref{Lemma2.2}.

\begin{lemma} \label{Lemma2.1}

 The rate function $I(f)$ (defined in  (\ref{N3.1})) is a good rate function, i.e.
for any $r\ge 0$  the set
$$
  B_r:=\{f\in \C:~I(f)\le r\}
$$
is a compact in $\C$ and
\begin{equation}{\label{0.7}}
  \lim _{\v\to 0}I((f)_\v)=I(f).
\end{equation}
\end{lemma}

\begin{proof}

First we show that the limit exists.
It is known (see e.g. \cite{SMG1}, Theorem 3.1 or Lemma 1.3),
that LDP implies local LDP: for any
$f\in \C[0,T]$
$$
  -I_0^T(f)\ge \lim_{\v\to 0}
  \lms_{n\to \infty}\frac{1}{n}\ln {\bf P}(X^{(T)}_n\in (f)_{T,\v})\ge
  \lim_{\v\to 0}
  \lmi_{n\to \infty}\frac{1}{n}\ln {\bf P}(X^{(T)}_n\in (f)_{T,\v})\ge
  -I_0^T(f).
$$
 For $U\ge T$,  with obvious notations,  we have for $f\in \C$
$$
  \{X^{(U)}_n\in (f^{(U)})_{U,\v}\}\subset  \{X^{(T)}_n\in (f^{(T)})_{T,\v}\},
$$
therefore
\begin{eqnarray*}
-I_0^T(f^{(T)})&\ge& \lim_{\v\to 0}
  \lms_{n\to \infty}\frac{1}{n}\ln {\bf P}(X^{(T)}_n\in (f^{(T)})_{T,\v})\\
  &\ge&\lim_{\v\to 0}
  \lmi_{n\to \infty}\frac{1}{n}\ln {\bf P}(X^{(U)}_n\in (f^{(U)})_{U,\v})\ge
  -I_0^U(f^{(U)}).
\end{eqnarray*}
Thus we established that   $I_0^T(f^{(T)})$ is non-decreasing in $T$, and \eqref{N3.1} follows. \\

Next, we show lower semi-continuity  (\ref{0.7}), that is  if $f_n\to f$, then
\begin{equation}{\label{0.8}}
\lmi_{n\to \infty}I(f_n)\ge I(f).
\end{equation}
For any $N<\infty$, $\v>0$ there is $T=T_{N,\v}<\infty$
such that
$$
  I_0^T(f)\ge \min\{I(f),N\}-\v.
$$
Since $f_n\to f$,   $\r_T(f_n,f)\to 0$.
The rate function $I_0^T(f)$ is lower semi-continuous in

$(\C[0,T],\r_T)$ (due to condition ${\bf I}$),
therefore
$$
\lmi_{n\to \infty}I(f_n)\ge \lmi_{n\to \infty}I_0^T(f_n)\ge
I_0^T(f)\ge \min\{I(f),N\}-\v.
$$
Since $N<\infty$ and $\v>0$
are arbitrary, the latter implies
(\ref{0.8}).

          We show next that the set $B_r$
          is completely bounded. For any $\v>0$  due to condition ${\bf II}$ there is $T=T_r<\infty$ such that
           for any $f\in B_r$
          \begin{equation}{\label{0.9}}
             \sup_{t\ge T}\frac{|f(t)|}{1+t^{1+\kappa}}<\v.
        \end{equation}
Denote
$$
  B_r^{(T)}:=\{f^{(T)}:~f\in B_r\},
$$
so that
$$
    B_r^{(T)}\subset B_{T,r},
$$
where we recall that $B_{T,r}:=\{f\in \C[0,T]:~I_0^T(f)\le r\}$.

Since by ${\bf I}$ the set
$B_{T,r}$ is a compact in $\C[0,T]$,
it is possible to find finite $\v$-net:
$$
  B_r^{(T)}\subset B_{T,r}\subset \cup_{i=1}^M (f_i)_{T,\v}.
$$
Now for $f\in \C[0,T]$  define $f^{(T+)}\in \C_0$
as
$$
f^{(T+)}(t):=
\left\{
           \begin{array}{lcl}
                              f(t), ~~\mbox{if} ~~ 0\le t\le T,\\
                              f(T), ~~\mbox{if} ~~ t\ge T\\
                              \end{array}
                              \right.
$$
 For any $f\in B_r$
 there is $i\in \{1,\cdots,M\}$ such that
 $$
   \sup_{0\le t\le T}\frac{|f(t)-f^{(T+)}_i(t)|}{1+t^{1+\kappa}}<\v<3\v.
 $$
 We have for this $i$ due to (\ref{0.9})
 $$
 \sup_{t\ge T}\frac{|f(t)-f^{(T+)}_i(t)|}{1+t^{1+\kappa}}\le
 \sup_{t\ge T}\frac{|f(t)|}{1+t^{1+\kappa}}+
 \sup_{t\ge T}\frac{|f(T)|}{1+t^{1+\kappa}}+
 \sup_{t\ge T}\frac{|f(T)-f^{(T+)}_i(T)|}{1+t^{1+\kappa}}\le 3\v,
 $$
therefore the collection $\{f^{(T+)}_1,\cdots,f^{(T+)}_M\}$
represents a $3\v$-net in the set $B_r$.
Thus we have shown that the set  $B_r$
 is completely bounded in $\C_0$.

 From lower semi-continuity of
 $I(f)$, established earlier, it follows  that $B_r$ is closed in $\C_0$. Since
 {\it a closed completely bounded subset of a Polish space is a compact} (see \cite{K-F}, Theorem 3, p. 109), we have shown that $B_r$ is a compact in $\C_0$,
thus completing the proof of Lemma \ref{Lemma2.1}.
\end{proof}

\begin{lemma}
\label{Lemma2.2}
 For any $f\in \C_0$, $\v>0$
          \begin{equation}{\label{0.13}}
              \lms_{n\to \infty}\frac{1}{n}\ln {\bf P}(X_n\in (f)_\v)\le
              -I((f)_{2\v}),
        \end{equation}
         \begin{equation}{\label{0.14}}
       \lmi_{n\to \infty}\frac{1}{n}
        \ln {\bf P}(X_n\in (f)_\v)\ge -I(f).
        \end{equation}
\end{lemma}

\begin{proof} $(i)$. First we prove the lower bound (\ref{0.14}) as it is also used in the proof of the upper bound.
If $I(f)=\infty$, then (\ref{0.14}) is trivially satisfied.
Let now  $I(f)<\infty$.
For any $T\in (0,\infty)$ there holds the inclusion
$$
  \{X_n\in (f)_\v\}\supset 
  A_n(T)\cap B_n(T)\cap  D(T),
$$
where
$$
  A_n(T):=\left\{\sup_{0\le t\le T}\frac{|X_n(t)-f(t)|}{1+t^{1+\kappa}}<\v\right\},~~~
B_n(T):=\left\{\sup_{ t\ge T}\frac{|X_n(t) |}{1+t^{1+\kappa}}<\v/4\right\},
$$
$$
D(T):=\left\{\sup_{ t\ge T}\frac{|f(t)|}{1+t^{1+\kappa}}<\v/2\right\}.
$$
For a large   $T$ the event $D(T)$
is a certainty (due to $I(f)<\infty$).
Therefore there exists $T_0<\infty$, such that for all $T\ge T_0$ it holds that
\begin{equation}{\label{0.15}}
{\bf P}(X_n\in (f)_\v)\ge {\bf P}(A_n(T)\cap B_n(T))\ge
{\bf P}(A_n(T))-{\bf P}(\overline{B_n}(T)),
        \end{equation}
where $\overline{B_n}(T)$ is a complement of $B_n(T)$.
Due to condition  ${\bf III}$ there is $T\ge T_0$
such  that
\begin{equation}{\label{0.16}}
\lms_{n\to \infty}\frac{1}{n}\ln {\bf P}(\overline{B_n}(T))\le -2I(f),
        \end{equation}
and for this $T$ due to  ${\bf I}$ we have
\begin{equation}{\label{0.17}}
\lmi_{n\to \infty}\frac{1}{n}\ln {\bf P}(A_n(T))\ge -I_0^T(f^{(T)})\ge -I(f).
        \end{equation}
(\ref{0.14})  now follows from (\ref{0.15}) by using   (\ref{0.16}), (\ref{0.17}).

$(ii)$. Now we prove the upper bound (\ref{0.13}).
It is obvious that for any $T\in (0,\infty)$
$$
  {\bf P}(X_n\in (f)_\v)\le  {\bf P}(X^{(T)}_n\in (f^{(T)})_{T,\v}),
$$
where \ we recall  \ that  $(f)_{T,\v}$
denote $\v$-neighbourhood in metric $\r_T$ in space $\C[0,T]$
of $f\in \C[0,T]$.

Due to condition ${\bf I}$ for any $\d>0$
\begin{eqnarray*}
L(\v):= \lms_{n\to \infty} \frac{1}{n}\ln {\bf P}(X_n\in (f)_\v)&\le&
\lms_{n\to \infty} \frac{1}{n}\ln {\bf P}(X^{(T)}_n\in (f^{(T)})_{T,\v})\\
&\le& -I_0^T([(f^{(T)})_{T,\v}]_T)\le -I_0^T((f^{(T)})_{T,\v+\d}).
\end{eqnarray*}
For any $T\in (0,\infty)$ and chosen $\v$  and $\d$, in this way we have the inequality
\begin{equation}{\label{0.18}}
L(\v)\le -I_0^T((f^{(T)})_{T,\v+\d}).
        \end{equation}
Choose now $T<\infty$ so large, that simultaneously the following holds:
\begin{equation}{\label{0.19}}
\sup_{t\ge T}\frac{|f(t)|}{1+t^{1+\kappa}}<\d;
        \end{equation}
\begin{equation}{\label{0.20}}
\lms_{n\to \infty}\frac{1}{n}\ln {\bf P}(X_n\in R(T,\v))\le -N,
        \end{equation}
where $N<\infty$ is arbitrary, and
$$
  R(T,\v):=\left\{g\in \C_0:~\sup_{t\ge T}\frac{|g(t)|}{1+t^{1+\kappa}}>\v\right\}.
$$
Denote
$$
  ({(f^{(T)})_{T,\v+\d}})^{(T+)}:=\{g\in \C_0:~g^{(T)}\in (f^{(T)})_{T,\v+\d}\}.
$$

Next we show that\begin{equation}{\label{0.4}}
I_0^T(B)=I(B^{(T+)}),
 \end{equation}
where  for $B\subset \C[0,T]$  $$
  B^{(T+)}:=\{g\in \C_0:~g^{(T)}\in B\}.
$$

Indeed, for any $\v>0$ let $f\in B$  be such that
$$
    I_0^T(f)\le I_0^T(B)+\v.
$$
Then due to  (\ref{0.3}) in condition ${\bf I}$
there is $g\in \C_0$ such that $g^{(T)}=f$
(consequently $g\in B^{(T+)}$) with
$I(g)=I_0^T(f)$.
Therefore
$$
   I_0^T(B)+\v\ge I_0^T(f)=I(g) \ge I(B^{(T+)}).
$$
Since $\v>0$  is arbitrary,
\begin{equation}{\label{0.5}}
I_0^T(B)\ge I(B^{(T+)}).
        \end{equation}
Let now $g\in B^{(T+)}$ such that
$$
    I(g)\le I(B^{(T+)})+\v.
$$
Then $g^{(T)}\in B$ with $I_0^T(g^{(T)})\le I(g)$.
Therefore
$$
  I(B^{(T+)})+\v\ge  I(g)\ge  I_0^T(g^{(T)})\ge I_0^T(B),
$$
and
\begin{equation}{\label{0.6}}
I_0^T(B)\le I(B^{(T+)}).
        \end{equation}
Inequalities (\ref{0.5}),  (\ref{0.6})
now prove equality (\ref{0.4}).\\

Due to  (\ref{0.4}) we have
$$
 I_0^T((f^{(T)})_{T,\v+\d})=I(({(f^{(T)})_{T,\v+\d}})^{(T+)}),
$$
therefore due to (\ref{0.18})
\begin{equation}{\label{0.21}}
L(\v)\le -I(({(f^{(T)})_{T,\v+\d}})^{(T+)}).
        \end{equation}
Take an arbitrary $g\in ({(f^{(T)})_{T,\v+\d}})^{(T+)}$. Then either
$$
  \sup_{t\ge T}\frac{|g(t)-f(t)|}{1+t^{1+\kappa}}<\v+2\d,
$$
and then
\begin{equation}{\label{0.22}}
  g\in (f)_{\v+2\d};
\end{equation}
or
\begin{equation}{\label{0.23}}
  \sup_{t\ge T}\frac{|g(t)-f(t)|}{1+t^{1+\kappa}}\ge\v+2\d,
\end{equation}
and then
\begin{equation}{\label{0.24}}
\sup_{t\ge T}\frac{|g(t)|}{1+t^{1+\kappa}}\ge \v+\d,
 \end{equation}
and
\begin{equation}{\label{0.25}}
  g\in R(T,\v).
\end{equation}
To clarify deduction of  (\ref{0.24}) from  (\ref{0.23}), note that if the inequality
(\ref{0.24}) is not true, then the opposite holds
$$
\sup_{t\ge T}\frac{|g(t)|}{1+t^{1+\kappa}}< \v+\d,
$$
and due to   (\ref{0.19})
$$
\sup_{t\ge T}\frac{|g(t)-f(t)|}{1+t^{1+\kappa}}\le
\sup_{t\ge T}\frac{|g(t)|}{1+t^{1+\kappa}}+
\sup_{t\ge T}\frac{|f(t)|}{1+t^{1+\kappa}}<\v+\d+\d=\v+2\d,
$$
which contradicts  (\ref{0.23}).
We have proved (see (\ref{0.22}) and
(\ref{0.25})),
 that
$$
  ({(f^{(T)})_{T,\v+\d}})^{(T+)}\subset (f)_{\v+2\d}\cup R(T,\v).
$$
From the latter we obtain
\begin{equation}{\label{0.26}}
I(({(f^{(T)})_{T,\v+\d}})^{(T+)})\ge \min\{I((f)_{\v+2\d}),I(R(T,\v))\}.
\end{equation}
Further, due to  (\ref{0.20})
$$
 -N\ge \lms_{n\to \infty}\frac{1}{n}\ln {\bf P}(X_n\in R(T,\v))\ge
\lmi_{n\to \infty}\frac{1}{n}\ln {\bf P}(X_n\in R(T,\v))\ge
-I(R(T,\v)),
$$
where the last inequality for an open set
$R(T,\v)$
follows from the established lower bound (\ref{0.14}).
Therefore
$$
  I(R(T,\v))\ge N,
$$
and, in view of  (\ref{0.26}),
$$
I(({(f^{(T)})_{T,\v+\d}})^{(T+)})\ge \min\{I((f)_{\v+2\d}),N\}.
$$
Going back to (\ref{0.21}), we obtain the inequality
$$
 L(\v)\le  -\min\{I((f)_{\v+2\d}),N\},
$$
in which $\d>0$  and $N<\infty$ are arbitrary.
Taking $2\d=\v$ and sending $N$ to $\infty$,
we obtain the required upper bound
$$
   L(\v)\le  -I((f)_{2\v}).
$$
Lemma \ref{Lemma2.2} is now proved.
\end{proof}

Local LDP for $\{X_n\}$ in $\C_0$ follows from
\ref{Lemma2.2}, and is stated as a corollary.

\begin{corollary} For any $f\in \C_0$
         $$
             \lim_{\v\to 0} \lms_{n\to \infty}\frac{1}{n}\ln {\bf P}(X_n\in (f)_\v)\le
              -I(f),
        $$
        $$
      \lim_{\v\to 0} \lmi_{n\to \infty}\frac{1}{n}
        \ln {\bf P}(X_n\in (f)_\v)\ge -I(f).
       $$
\end{corollary}

Next result proves a weaker form of exponential tightness: for any $N$ there is a completely bounded set $K_N$ in  $(\C_0,\r)$ such that
 $$
   \lms_{n\to \infty}\frac{1}{n}\ln {\bf P}(X_n\not \in K)\le -N.
 $$

\begin{lemma} \label{Lemma2.3}
For any $N<\infty$ and $\v>0$
there is a finite collection of  $g_1,\cdots, g_M\in \C_0$
such that
 $$
             \lms_{n\to \infty}\frac{1}{n}
             \ln {\bf P}(X_n\not\in \cup_{i=1}^M(g_i)_\v)\le -N.
   $$
\end{lemma}

\begin{proof} Denote by
$$
  R_T(\v):=\left\{f\in \C_0:~\sup_{t\ge T}\frac{|f(t)|}{1+t^{1+\kappa}}\le \v\right\}.
$$
Then due to condition  ${\bf III}$
there is $T<\infty$ such that
\begin{equation}{\label{0.27}}
\lms_{n\to \infty}\frac{1}{n}\ln {\bf P}(X_n\not \in R_T(\v))\le -N.
        \end{equation}
For this $T$ due to condition ${\bf I}$ the process
$X^{(T)}_n$ satisfies LDP in the space $\C[0,T]$.
Therefore for a chosen $N$ by a theorem of Puhalskii
(see  \cite{Puh2}  {\bf page or Theorem number} )
there is a compact $K\subset \C[0,T]$
such that
$$
  \lms_{n\to \infty}\frac{1}{n}\ln {\bf P}(X^{(T)}_n\not \in K)\le -N.
$$
For a given $\v>0$
take a finite $\v$-net $f_1,\cdots,f_M\in \C[0,T]$
in $K$:
$$
  K\subset \cup_{i=1}^M(f_i)_{T,\v}.
$$
Then
\begin{equation}{\label{0.28}}
                          \lms_{n\to \infty}\frac{1}{n}
             \ln {\bf P}(X^{(T)}_n\not \in \cup_{i=1}^M(f_i)_{T,\v})\le -N.
        \end{equation}
Denote for all  $i=1,\cdots,M$
$$
  {g}_i(t):=
  \left\{
           \begin{array}{lcl}
                              f_i(t), ~~\text{if} ~~ 0\le t\le T,\\
                              f_i(T), ~~\text{if} ~~ t\ge T.\\
                              \end{array}
                              \right.
$$
Define the set ${\mathcal M}_\v:=\{i\in \{1,\cdots,M\}:~\frac{|f_i(T)|}{1+T^{1+\kappa}}\le 2\v\}$.
Then
$$
 P:={\bf P}(X_n\not \in \cup_{i=1}^M({g}_i)_{3\v}) \le
 {\bf P}(X_n\not \in R_T(\v))+{\bf P}(X_n \in R_T(\v),~X^{(T)}_n\not \in \cup_{i=1}^M(f_i)_{T,\v})+
 $$
 $$
 {\bf P}(X_n \in R_T(\v),~X^{(T)}_n\in \cup_{i=1}^M(f_i)_{T,\v},~~
 X_n\not \in \cup_{i=1}^M({g}_i)_{3\v})=:P_1+P_2+P_3.
$$
We bound  $P_3$ as follows:
$$
 P_3\le
 \sum_{i\in {\mathcal M}_\v}{\bf P}\left(\sup_{t\ge T}\frac{|X_n(t)-f_i(T)|}
 {1+t^{1+\kappa}}>3\v\right)
 \le M {\bf P}\left(\sup_{t\ge T}\frac{|X_n(t)|}{1+t^{1+\kappa}}>\v\right)\le
   M {\bf P}(X_n\not \in R_T(\v)).
$$
Since
$$
   P_1\le {\bf P}(X_n\not \in R_T(\v)),
$$
we obtain
\begin{equation}{\label{0.29}}
 P\le {\bf P}(X^{(T)}_n\not \in \cup_{i=1}^M(f_i)_{T,\v})+(M+1){\bf P}(X_n\not \in R_T(\v)).
        \end{equation}
Using bounds  (\ref{0.27}) and (\ref{0.28}) with (\ref{0.29}),
we obtain the required inequality
$$
  \lms_{n\to \infty}\frac{1}{n}\ln {\bf P}(X_n\not \in
  \cup_{i=1}^M({g}_i)_{3\v})\le -N.
$$
Lemma \ref{Lemma2.3} is now proved.
\end{proof}

\section{Large Deviations for Random Walks}

\subsection{Large Deviation Principle for Random Walks on half-line.}

Let  $\x$ be a non-degenerate random variable satisfying
  the following condition

$[{\bf C}_\infty]$. {\it For any $\l\in \R$}
$$
  {\bf E}e^{\l\x}<\infty.
$$

Denote
$$
  \Lambda(\a):=\sup_\l \{\l\a-A(\l)\},~~~A(\l):=\ln {\bf E}e^{\l\x},
$$
the deviation function of $\x$.
It is a convex non-negative lower-semiconscious function with a single zero at $\a={\bf E}\x$, (see e.g.
\cite{BB} or \cite{BBM}).

Denote
$$
  S_0:=0,~~~S_k:=\x_1+\cdots+\x_k~~~\mbox{for}~~~k\ge 1,
$$
where $\{\x_n\}$ is a sequence of i.i.d. copies of $\x$. Consider a random piece-wise linear function $s_n=s_n(t)\in \C$,
 going through the nodes
 $$
   \left({\frac{k}{n},~\frac{S_k}{x}}\right),~~~k=0,1,\cdots,
 $$
 where $x=x(n)$ is a fixed sequence of positive constants such that $x\sim n$ as $n\to \infty$.
 The rate function corresponding to the process $s_n$ is defined as
 $$
  I(f):=
  \left\{
           \begin{array}{lcl}
                              \int_0^\infty  \Lambda(f'(t)) dt, ~~ ~~~f(0)=0,~~~f~~ \text{is absolutely continuous},\\
                              +\infty, \text{otherwise.}\\
                              \end{array}
                              \right.
$$

\begin{theorem} \label{thm3.1} Assume $[{\bf C}_\infty]$.
 Then $s_n$ satisfies LDP in space
 $(\C,~\r_\kappa)$ for $\kappa=0$ with rate function $I$.
 \end{theorem}

%
%

\begin{proof}  Without loss of generality we can take ${\bf E}\x=0$. This is because  the deviation function for $\x^{(0)}:=\x-a$
 is given by
 $
   \Lambda^{(0)}(\a)=\Lambda(\a+a).
 $ ( superscript $^{(0)}$ denotes quantities for
 the centered random variable).
 Therefore the rate function for $s^{(0)}_n$, is given by
 $
  I^{(0)}(f)=I(f+e_a),
 $
 where $e_a=e_a(t):=at;~t\ge 0$. Clearly, $s_n=s^{(0)}_n+e_a$,
 $
  {\bf P}(s_n\in B)={\bf P}(s^{(0)}_n\in B-e_a),
 $
where $B-e_a:=\{f-e_a:~f\in B\}$. It is obvious that
 $
  [B-e_a]=[B]-e_a,~~~(B-e_a)=(B)-e_a,
 $
implying
 $
  I^{(0)}([B-e_a])=I([B]),~~~I^{(0)}((B-e_a))=I((B)).
  $
Hence  the LDP for $s^{(0)}_n$  with rate function
$I^{(0)}$  implies LDP for  $s_n$   with rate function $I$.

The rest of the proof
 consists in checking conditions ${\bf I}-{\bf III}$ of Theorem \ref{thm}.
Condition ${\bf I}$ follows from the LDP for
$s_n$ in  $\C[0,1]$ (see \cite{BBM}, Theorem 9 or \cite{SMG1}, Section 6.2).

Proof of ${\bf II}$.
By $[{\bf C}_\infty]$, with ${\bf E}\x=0$ it follows that there exists a   non-decreasing continuous function
$h(t);~t\ge 0$, such that for some $\d>0$, $h(t)=\d t$, if $0\le t\le 1$,
$\lim_{t\to \infty}h(t)=\infty$, and that for all
$\a\in \R$ the following inequality holds
\begin{equation}{\label{1.1}}
\Lambda(\a)\ge h(|\a|)|\a|.
\end{equation}

Indeed, for $\a\to 0$   (see e.g.
 \cite{BBM}, p.21)
\begin{equation}{\label{1.2}}
\Lambda(\a)\sim \frac{\a^2}{2\sigma^2};
\end{equation}
for any $\l>0$, $\a>0$
$$
  \Lambda(\a)\ge \l\a-A(\l),~~~\Lambda(-\a)\ge \l\a-A(-\l).
$$
Therefore
$$
  \lmi_{|\a|\to \infty}\frac{\Lambda(\a)}{|\a|}\ge \l,
$$
so that
\begin{equation}{\label{1.3}}
\lmi_{|\a|\to \infty}\frac{\Lambda(\a)}{|\a|}=\infty.
\end{equation}
(\ref{1.1}) follows from (\ref{1.2}) and (\ref{1.3}).

Denote by
$$
  f_T(t):=t\frac{f(T)}{T},~~~t\in [0,T].
$$
The function $f_T$ ``straightens" function $f$ on $[0,T]$:
$$
  I_0^T(f)\ge I_0^T(f_T)=\int_0^T\Lambda\left(\frac{f(T)}{T}\right)dt=
  T\Lambda\left(\frac{f(T)}{T}\right).
$$
Therefore by (\ref{1.1}) for $f\in B_r$
$$
  r\ge T\Lambda\left(\frac{f(T)}{T}\right)\ge T \frac{|f(T)|}{T}h\left(\frac{|f(T)|}{T}\right),
  $$
so that
\begin{equation}{\label{1.4}}
   \frac{|f(T)|}{T}\le \frac{r}{Th\left(\frac{|f(T)|}{T}\right)}.
\end{equation}
Let $c:=\sqrt{\frac{r}{\d}}$, and $T$ be such that $\frac{c}{\sqrt{T}}\le 1$.
Assume that
$$
  \frac{|f(T)|}{T}>\frac{c}{\sqrt{T}}.
$$
Then it follows from $(\ref{1.4})$
$$
  \frac{|f(T)|}{T}\le \frac{r}{Th(\frac{c}{\sqrt{T}})}=\frac{r}{T\d\frac{c}{\sqrt{T}}}=\frac{c}{\sqrt{T}},
$$
which is a contradiction.
Thus for
  $T\ge c^2=\frac{r}{\d}$ it holds
$$
  |f(T)|\le \sqrt{\frac{r}{\d}}\sqrt{T}.
$$
Clearly,  $B_r={B}^+_r$. Therefore we have proved
$$
  \sup_{f\in B^+_r} \sup_{t\ge T}\frac{|f(t)|}{1+t}\le
  \sqrt{\frac{r}{\d}}\frac{\sqrt{T}}{1+T}\le \sqrt{\frac{r}{\d}}\frac{1}{\sqrt{T}}.
$$
Condition ${\bf II}$   now follows.

Check now condition ${\bf III}$. For $T_n:=\max\{\frac{k}{n}\le T:~k=1,2,\cdots\}$,
we have
\begin{eqnarray*}
  {\bf P}\left(\sup_{t\ge T}\frac{|s_n(t)|}{1+t}>\v\right)&\le&
  {\bf P}\left(\sup_{t\ge T_n}\frac{|s_n(t)|}{1+t}>\v\right)\\
  &\le&{\bf P}\left(\sup_{t\ge T_n}\frac{|s_n(t)-s_n(T_n)|}{1+t}>\v/2\right)+
  {\bf P}\left(\sup_{t\ge T_n}\frac{|s_n(T_n)|}{1+t}>\v/2\right)\\
&=&
 {\bf P}\left(\sup_{u\ge 0}\frac{|s_n(u)|}{1+T_n+u}>\v/2\right)+
  {\bf P}\left(\sup_{t\ge T_n}\frac{|s_n(T_n)|}{1+T_n}>\v/2\right)\\
  &\le&
  {\bf P}\left(\sup_{k\ge 1}\frac{|s_n(\frac{k}{n})|}{T+\frac{k}{n}}>\v/2\right)+
  {\bf P}\left(\frac{|s_n(T_n)|}{T_n}>\v/2\right)=:P_1(n)+P_2(n).
\end{eqnarray*}
To bound $P_1(n)$ use the exponential Chebyshev's (Chernoff's) inequality
(see e.g. \cite{BB} or \cite{BBM}):
\begin{eqnarray*}
  P_1(n) &\le& \sum_{k\ge 1}{\bf P}\left(\frac{|S_k|}{x(T+\frac{k}{n})}>\v/2\right)\\
  &\le&
  \sum_{k\ge 1}{\bf P}\left(\frac{S_k}{x(T+\frac{k}{n})}>\v/2\right)+
  \sum_{k\ge 1}{\bf P}\left(\frac{S_k}{x(T+\frac{k}{n})}<-\v/2\right)\\
   &\le&
   \sum_{k\ge 1}e^{-k\Lambda(R)}+ \sum_{k\ge 1}e^{-k\Lambda(-R)},
\end{eqnarray*}
   where $R:=\frac{x}{k}(T+\frac{k}{n})$.
   Since for all $n$  and $T$ large enough
   $$
     R\ge \v/4,~~~kR\ge T\v/4+k\v/4,
   $$
   we have due to (\ref{1.1})  for $\v/4\in (0,1)$
   $$
     k\Lambda(\pm R)\ge k Rh(R)\ge (T\v/4+k\v/4)\d\v/4=T\frac{\d\v^2}{16}+k\frac{\d\v^2}{16}.
        $$
 Therefore
 $$
   P_1(n) \le 2 e^{-T\d_1}\sum_{k\ge 1}e^{-k\d_1}=C_1e^{-T\d_1},
 $$
where $\d_1:=\frac{\d\v^2}{16}$, ~~$C_1:=2\frac{e^{-\d_1}}{1-e^{-s_1}}$.

Similarly we obtain the bound
$$
 P_2(n) \le C_2 e^{-T\d_2}
$$
for some $\d_2>0$,
$C_2<\infty$.
Hence condition ${\bf III}$ holds  and   the proof   is complete.
\end{proof}

\bigskip


\subsection{Moderate Deviation Principle for Random Walks on half-line.}
Let random piece-wise linear  function $s_n=s_n(\cdot)\in \C$
be defined as before by the sums $S_k$  of independent random variables  distributed as $\x$. Let
$\x$ have zero mean ${\bf E}\x=0$  and assume Cramer's condition

$[{\bf C_0}].$ {\it For some $\d>0$}
$$
  {\bf E}e^{\d|\x|}<\infty.
$$
Let a sequence $x=x(n)$, used in the construction of
$s_n$,
satisfy
$$
\frac{x}{\sqrt{n}}\to \infty,~~~\frac{x}{n}\to 0~~~\mbox{if}~~~n\to \infty.
$$
The rate function for   $s_n$  is defined as
$$
  I_0(f):=
  \left\{
           \begin{array}{lcl}
                              \frac{1}{2\sigma^2}\int_0^\infty  (f'(t))^2 dt,~~~
  \mbox{if}~~~f(0)=0,~~~f~~\mbox{is absolutely continuous}\\
  \infty~~~\mbox{otherwise},\\
                              \end{array}
                              \right.
$$
where $\sigma^2:={\bf E}\x^2$.

\begin{theorem} Let ${\bf E}\x=0$ and
 condition $[{\bf C}_0]$ holds.
 Then $s_n$ satisfies LDP
 with speed $\frac{x^2}{n}$ and rate function $I_0$
 in space
 $(\C,~\r_\kappa)$ with $\kappa=0$, i.e. for any measurable set
 $B\subset \C$
 $$
   \lms_{n\to \infty}\frac{n}{x^2}\ln{\bf P}(s_n\in B)\le -I_0([B]),
 $$

 $$
   \lmi_{n\to \infty}\frac{n}{x^2}\ln{\bf P}(s_n\in B)\ge -I_0((B)).
 $$
\end{theorem}
Similarly to the proof of Lemma 3.1, the proof of Theorem   3.2 consists in checking   conditions ${\bf I}-{\bf III}$, replacing $n$ by
$\frac{x^2}{n}$. In all other details the proof is the same.

Condition ${\bf I}$ is verified with help of \cite{M76} (Theorem 1) or \cite{BMUBU} (Theorem 2.2).
Condition ${\bf II}$  is obvious. Only condition
${\bf III}$ requires a clarification, which is done by using the following form of Kolmogorov's inequality (\cite{B}, p. 295, lemma 11.2.1):

\begin{lemma} For any $x\ge 0$, $y\ge 0$, $n\ge 1$
$$
{\bf P}(\max_{1\le m\le n}|S_m|\ge x+y)\le \frac{{\bf P}(|S_n|\ge x)}
{\min\limits_{1\le m\le n}{\bf P}(|S_m|\le y)}.
$$
\end{lemma}
\begin{proof}
An upper bound for
$$
P:=  {\bf P}\left(\sup_{t\ge T}\frac{|s_n(t)|}{1+t}\ge \v\right)
$$
is obtained by using
\begin{eqnarray*}
  P &\le& \sum_{K\ge T}{\bf P}\left(\sup_{K\le t\le K+1}|s_n(t)|\ge K\v\right)\\
  &\le&
  \sum_{K\ge T}{\bf P}(\sup_{K\le t\le K+1}|s_n(t)-s_n(K)|\ge K\v/2)+
  \sum_{K\ge T}{\bf P}(|s_n(K)|\ge K\v/2),
\end{eqnarray*}
so that
\begin{equation}{\label{2.1}}
P\le \sum_{K\ge T}P_1(K,n)+\sum_{K\ge T}P_2(K,n),
\end{equation}
where
$$
P_1(K,n):={\bf P}(\sup_{0\le t\le 1}|s_n(u)|\ge K\v/2),~~~
P_2(K,n):={\bf P}(|s_n(K)|\ge K\v/2).~~~
$$

We bound $P_2(K,n)$ using exponential Chebyshev's inequality:
$$
  P_2(K,n)={\bf P}\left(\frac{|S_{nK}|}{nK}\ge \frac{x\v}{2n}\right)\le
  e^{-nK\Lambda(\frac{x\v}{2n})}+e^{-nK\Lambda(-\frac{x\v}{2n})}.
$$
Since for all $n$  large enough
$\frac{x\v}{2n}\le 1$, then by (\ref{1.1})
$$
  nK\Lambda\left(\pm\frac{x\v}{2n}\right)\ge \frac{x^2}{n}K \d_1,~~~
  \d_1:=\frac{\d\v^2}{4}.
$$
Therefore
\begin{equation}{\label{2.2}}
\sum_{K\ge T}P_2(K,n)\le 2 \frac{e^{-\frac{x^2}{n}T\d_1}}{1-e^{-\frac{x^2}{n}\d_1}}.
\end{equation}

Bound now $P_1(K,n)$ using Kolmogorov's inequality (Lemma 3.2) and exponential Chebyshev's inequality (Chernoff)
$$
 P_1(K,n)={\bf P}\left(\max_{1\le m\le n}\frac{|S_m|}{xK}\ge \v/4+\v/4\right)\le
 \frac{1}{c}{\bf P}\left(\frac{|S_n|}{xK}\ge \v/4\right),
$$
where
$$
  c:=\min_{1\le m\le n}{\bf P}\left(\frac{|S_m|}{xK}<\v/4\right).
$$
Since
$$
c=\min_{1\le m\le n}{\bf P}\left(\frac{|S_m|}{xK}< \v/4\right)\ge
\min_{1\le m\le n}{\bf P}\left(\frac{|S_m|}{\sqrt{m}}< \frac{xT}{\sqrt{n}}\v/4\right)\to 1
$$
as $n\to \infty$,
for $n$ large enough
$$
 P_1(K,n)\le 2{\bf P}\left(\frac{|S_n|}{xK}\ge \v/4\right)\le
 2e^{-n\Lambda(\frac{xK}{n}\v/4)}+2e^{-n\Lambda(-\frac{xK}{n}\v/4)}.
$$
By (\ref{1.1})  for large enough $n$  and some $\d_1>0$
we have
$$
  n\Lambda\left(\pm\frac{xK}{n}\v/4\right)\ge \frac{x^2}{n}K\d_1,
$$
therefore
$$
   P_1(K,n)\le 4 e^{-\frac{x^2}{n}K\d_1},
$$
\begin{equation}{\label{2.3}}
\sum_{K\ge T}P_1(K,n)\le 4 \frac{e^{-\frac{x^2}{n}T\d_1}}{1-e^{-\frac{x^2}{n}\d_1}}.
\end{equation}
Applying (\ref{2.2}), (\ref{2.3})  to
(\ref{2.1}), we have for $T\ge \frac{N}{\d_1}$
the required inequality in ${\bf III}$:
$$
  \lms_{n\to \infty}\frac{n}{x^2}\ln{\bf P}\le -N.
$$
Thus condition  ${\bf III}$ is proved.
\end{proof}

\section{Large Deviations for Diffusion Processes on half-line}

\subsection{Zero drift.}
Consider a stochastic process $X_n(t)$, $t\geq 0$, defined on the stochastic basis
$(\Omega,\mathfrak{F},\mathfrak{F}_t,\bf P)$ that is an It\^o integral with respect to Wiener process $w(t)$.
$$
X_n(t)=x_0+\frac{1}{\sqrt{n}}\int_0^t \sigma_n(\omega,s)dw(s),
$$
where  $\sigma_n(\omega,t)$  is $\mathfrak{F}_t$-adapted and such that the It\^o integral is defined.
%

\begin{lemma}
 Let for some  $\lambda>0$ and all $t\geq 0, \ n\geq 0$
\begin{equation}{\label{4.1}}
\sigma_n^2(\omega,t)\leq \lambda \  \text{a.s.}
\end{equation}
Then for any $N<\infty$ and $\varepsilon>0$ there exists $T=T_{N,\varepsilon}<\infty$
such that
\begin{equation}{\label{4.2}}
\lms_{n\to \infty}\frac{1}{n}\ln \mathbf{P}\biggl(
\sup\limits_{t\geq T}\frac{|X_n(t)|}{1+t}>\varepsilon\biggl)\leq -N.
\end{equation}
\end{lemma}

\begin{proof}
For $T>1\vee(\frac{2|x_0|}{\varepsilon}-1)$ we have
\begin{eqnarray*}\label{4.3}
\mathbf{P}\biggl(\sup_{t\geq T}\frac{|X_n(t)|}{1+t}>\varepsilon\biggl)&\leq&
\mathbf{P}\biggl(\sup_{t\geq T}\frac{1}{1+t}\biggl|\frac{1}{\sqrt{n}}\int_0^t \sigma_n(\omega,s)dw(s)\biggl|>\frac{\varepsilon}{2}\biggl)\nonumber\\
&\leq&
 \mathbf{P}\biggl(\sup_{t\geq T}\frac{1}{t}\biggl|\int_0^t \sigma_n(\omega,s)dw(s)\biggl|>\frac{\sqrt{n}\varepsilon}{2}\biggl)\nonumber\\
 &=&
  \mathbf{P}\biggl(\bigcup_{r=1}^\infty\biggl\{\sup_{t\in [Tr,T(r+1))}\frac{1}{t}\biggl|
\int_0^t \sigma_n(\omega,s)dw(s)\biggl|>\frac{\sqrt{n}\varepsilon}{2}\biggl\}\biggl)\nonumber\\
&\leq&
\sum_{r=1}^\infty \mathbf{P}\biggl(\sup_{t\in [0,T(r+1)]}\biggl|
\int_0^t \sigma_n(\omega,s)dw(s)\biggl|>\frac{Tr\sqrt{n}\varepsilon}{2}\biggl)=:\sum_{r=1}^\infty P_r.
\end{eqnarray*}
We bound $P_r$ from above as follows. For any $c>0$ we have
\begin{eqnarray*}\label{4.4}
P_r&=& \mathbf{P}\biggl(\sup_{t\in [0,T(r+1)]}\exp\biggl\{\biggl|
c\int_0^t \sigma_n(\omega,s)dw(s)\biggl|\biggl\}>\exp\biggl\{\frac{cTr\sqrt{n}\varepsilon}{2}\biggl\}\biggl)\nonumber\\
&\leq&
  \mathbf{P}\biggl(\sup_{t\in [0,T(r+1)]}\exp\biggl\{
c\int_0^t \sigma_n(\omega,s)dw(s)\biggl\}>\exp\biggl\{\frac{cTr\sqrt{n}\varepsilon}{2}\biggl\}\biggl)\nonumber\\
&&+
  \mathbf{P}\biggl(\sup_{t\in [0,T(r+1)]}\exp\biggl\{-
c\int_0^t \sigma_n(\omega,s)dw(s)\biggl\}>\exp\biggl\{\frac{cTr\sqrt{n}\varepsilon}{2}\biggl\}\biggl)=:P_{r,1}+P_{r,2}.
\end{eqnarray*}
We proceed to bound $\mathbf{P}_{r,1}.$ For ease of notation we drop arguments in $\sigma_n(\omega,t)$.

Using (\ref{4.1}) we have
\begin{eqnarray*}
P_{r,1}& = &\mathbf{P}\biggl(\sup_{t\in [0,T(r+1)]}\exp\biggl\{
c\int_0^t \sigma_n dw(s)\pm \frac{c^2}{2}\int_0^t \sigma_n^2 ds\biggl\}>
\exp\biggl\{\frac{cTr\sqrt{n}\varepsilon}{2}\biggl\}\biggl)\\
&\leq &\mathbf{P}\biggl(\sup_{t\in [0,T(r+1)]}\exp\biggl\{
c\int_0^t \sigma_n dw(s)- \frac{c^2}{2}\int_0^t \sigma_n^2 ds\biggl\}>
\exp\biggl\{\frac{cTr\sqrt{n}\varepsilon-c^2\lambda T(r+1)}{2}\biggl\}\biggl).
\end{eqnarray*}
By Doob's martingale inequality  for
$$
M(t)=\exp\biggl\{
c\int_0^t \sigma_n(\omega,s) dw(s)- \frac{c^2}{2}\int_0^t \sigma_n^2(\omega,s) ds\biggl\},
$$
we have
$$
P_{r,1}\leq \frac{\mathbf{E}M(t)}{\exp\biggl\{\frac{cTr\sqrt{n}\varepsilon-c^2\lambda T(r+1)}{2}\biggl\}}=
\exp\biggl\{-\frac{cTr\sqrt{n}\varepsilon-c^2\lambda T(r+1)}{2}\biggl\}.
$$
Taking $c=\frac{\sqrt{n}\varepsilon T r }{2\lambda T(r+1)}$ we obtain
\begin{equation}{\label{4.5}}
P_{r,1}\leq \exp\biggl\{-\frac{(Tr)^2n\varepsilon^2}{8\lambda T(r+1)}\biggl\}\leq
\exp\biggl\{-\frac{T r n\varepsilon^2}{16\lambda}\biggl\}.
\end{equation}
  $P_{r,2}$ is bounded in exactly the same way. Collecting the terms it now follows
\begin{eqnarray}\label{4.7}
\mathbf{P}\biggl(\sup_{t\geq T}\frac{|X_n(t)|}{1+t}>\varepsilon\biggl)&\leq & 2 \sum_{r=1}^\infty
\exp\biggl\{-\frac{T r n\varepsilon^2}{16\lambda}\biggl\}\nonumber\\
&=& \frac{2\exp\biggl\{-\frac{Tn\varepsilon^2}{16\lambda}\biggl\}}{1-\exp\biggl\{-\frac{Tn\varepsilon^2}{16\lambda}\biggl\}}\leq
4\exp\biggl\{-\frac{Tn\varepsilon^2}{16\lambda}\biggl\}.
\end{eqnarray}
Using inequality (\ref{4.7}) we have
$$
\lms_{n\to \infty}\frac{1}{n}\ln \mathbf{P}\biggl(\sup_{t\geq T}\frac{|X_n(t)|}{1+t}>\varepsilon\biggl)\leq
-\frac{T\varepsilon^2}{16\lambda}.
$$
It follows now that for $T>1\vee(\frac{2|x_0|}{\varepsilon}-1)\vee\frac{16\lambda N}{\varepsilon^2}$
$$
\lms_{n\to \infty}\frac{1}{n}\ln \mathbf{P}\biggl(\sup_{t\geq T}\frac{|X_n(t)|}{1+t}>\varepsilon\biggl)\leq -N,
$$
which proves (\ref{4.2}).
\end{proof}

Let now $X_n$ solve a stochastic differential equation (SDE)
\begin{equation}{\label{4.8}}
X_n(t)=x_0+\frac{1}{\sqrt{n}}\int_0^t \sigma(X_n(s))dw(s)
\end{equation}
on half-line $[0,\infty)$.

\begin{theorem} Let $\sigma(x)$ be a measurable function of real argument $x$, such that for some   $\lambda\geq 1$ and all $\ x\in R$
\begin{equation}{\label{4.9}}
\frac{1}{\lambda}\leq\sigma^2(x)\leq \lambda.
\end{equation}
Let the Lebesgue measure of discontinuities of $\sigma$ be zero.
Then $\{X_n\}$ satisfies LDP in space $(\mathbb{C},\rho_\kappa)$
with $\kappa = 0$ and good rate function:
$$ I(f)=\left\{
           \begin{array}{lcl}
                              \frac{1}{2}\int_0^\infty \frac{(f'(t))^2}{\sigma^2(f(t))}dt, \text{ if}\ f(0)=x_0, \
                              f  \text{ is absolutely continuous}, \\
                              \infty, \ \text{otherwise.}\\
                              \end{array}
                              \right.$$
\end{theorem}
\begin{proof}
Existence of weak solution in (\ref{4.8}) follows e.g. from Proposition 1 of \cite{Kul}.
Condition \textbf{I} holds by Theorem  1 in \cite{Kul} and by extending $f$ for $t\geq T$
by its value at  $T$.\\
Consider the rate function
$$ I_0^T(f^{(T)})=\left\{
           \begin{array}{lcl}
                              \frac{1}{2}\int_0^T \frac{(f'(t))^2}{\sigma^2(f(t))}dt, \text{ if}\ f(0)=x_0, \
                              f  \text{ is absolutely continuous}, \\
                              \infty, \ \text{otherwise.}\\
                              \end{array}
                              \right.$$

We verify condition \textbf{II}. Using (\ref{4.9}) and applying Cauchy-Bunyakovskii inequality, we have (recall that $B_r^+:=\{f\in \mathbb{C}: \ \lms_{T \to \infty}I_0^T(f^{(T)})\leq r\}$)
\begin{eqnarray*}
\lim\limits_{T\rightarrow \infty}\sup\limits_{f\in B_r^+}\sup\limits_{t\geq T}
\frac{|f(t)|}{1+t}&=&\lim\limits_{T\rightarrow \infty}\sup\limits_{f\in B_r^+}\sup\limits_{t\geq T}
\frac{1}{1+t}\biggl|\int_0^tf'(s)ds\biggl|
 \leq
  \lim\limits_{T\rightarrow \infty}\sup\limits_{f\in B_r^+}\sup\limits_{t\geq T}
\frac{t^{1/2}}{1+t}\biggl(\int_0^t(f'(s))^2ds\biggl)^{1/2}\\
&\leq& \lim\limits_{T\rightarrow \infty}\sup\limits_{f\in B_r^+}
\frac{1}{T^{1/2}}\biggl(\lambda\int_0^T\frac{(f'(s))^2}{\sigma^2(f(s))}ds\biggl)^{1/2}\\
&=&
\lim\limits_{T\rightarrow \infty}\sup\limits_{f\in B_r^+}
\frac{\sqrt{2\lambda}}{T^{1/2}}\sqrt{I_0^T(f^{(T)})}\leq
\lim\limits_{T\rightarrow \infty}
\frac{\sqrt{2\lambda r}}{T^{1/2}}=0.
\end{eqnarray*}

Condition \textbf{III} follows from
(\ref{4.9}) and  Lemma 4.1.
\end{proof}

\subsection{Non-zero drift.}
Consider solution of SDE on half-line $[0,\infty)$
\begin{equation}{\label{4.10}}
X_n(t)=x_0+\int_0^t a(X_n(s))ds+\frac{1}{\sqrt{n}}\int_0^t \sigma(X_n(s))dw(s).
\end{equation}
\begin{lemma}Suppose there exists $\lambda>0$  such that for all  $y\in R$
\begin{equation}{\label{4.11}}
|a(y)|+\sigma^2(y)\leq \lambda.
\end{equation}
Then for any $\kappa > 0$, $N<\infty$ and $\varepsilon>0$ there exists $T=T_{N,\varepsilon,\kappa}<\infty$
such that
\begin{equation}{\label{4.12}}
\lms_{n\to \infty}\frac{1}{n}\ln \mathbf{P}\biggl(
\sup\limits_{t\geq T}\frac{|X_n(t)|}{1+t^{1+\kappa}}>\varepsilon\biggl)\leq -N.
\end{equation}
\end{lemma}

\begin{proof}
Using condition (\ref{4.11}) for $T>1\vee(\frac{4|x_0|}{\varepsilon}-1)^\frac{1}{1+\kappa}\vee (\frac{4\lambda}{\varepsilon})^\frac{1}{k}$ we obtain
\begin{eqnarray*}
&&\mathbf{P}\biggl(\sup_{t\geq T}\frac{|X_n(t)|}{1+t^{1+\kappa}}>\varepsilon\biggl) \\
&\leq&
 \mathbf{P}\biggl(\sup_{t\geq T}\frac{1}{1+t^{1+\kappa}}\biggl(|x_0|+\biggl|\int_0^t a(X_n(s))ds\biggl|+\biggl|\frac{1}{\sqrt{n}}\int_0^t \sigma(X_n(s))dw(s)\biggl|\biggl)>\varepsilon\biggl)\\
 &\leq&
 \mathbf{P}\biggl(\sup_{t\geq T}\frac{1}{1+t^{1+\kappa}}\biggl(|x_0|+\lambda t+\biggl|\frac{1}{\sqrt{n}}\int_0^t \sigma(X_n(s))dw(s)\biggl|\biggl)>\varepsilon\biggl)\\
&\leq&
 \mathbf{P}\biggl(\sup_{t\geq T}\frac{1}{t^{1+\kappa}}\biggl|\int_0^t \sigma(X_n(s))dw(s)\biggl|>\frac{\sqrt{n}\varepsilon}{2}\biggl)
\leq \mathbf{P}\biggl(\sup_{t\geq T}\frac{1}{t}\biggl|\int_0^t \sigma(X_n(s))dw(s)\biggl|>\frac{\sqrt{n}\varepsilon}{2}\biggl).
\end{eqnarray*}
The rest of the proof repeats   that of   Lemma 4.1.
\end{proof}

\begin{theorem} Let $a(y)$ and $\sigma(y)$ are functions of real argument such that
for some $\lambda\geq 1$ and all $\ y,x \in R$
\begin{equation}{\label{4.13}}
|a(y)-a(x)|+|\sigma(y)-\sigma(x)|\leq \lambda|y-x|
\end{equation}
\begin{equation}{\label{4.14}}
\frac{1}{\lambda}\leq\sigma^2(y)\leq \lambda, \ |a(y)|\leq \lambda.
\end{equation}
Then for any given $\kappa > 0$ the sequence $X_n$ satisfies LDP in $(\mathbb{C},\rho_\kappa)$
 with good rate function:
$$ I(f)=\left\{
           \begin{array}{lcl}
                              \frac{1}{2}\int_0^\infty \frac{(f'(t)-a(f(t)))^2}{\sigma^2(f(t))}dt, \text{ if}\ f(0)=x_0, \
                              f  \text{ is absolutely continuous}, \\
                              \infty, \ \text{otherwise.}\\
                              \end{array}
                              \right.$$
\end{theorem}

\begin{proof}
Existence of a strong solution in (\ref{4.10}) is assured by Theorem 1 in \cite{Gih}.

Condition \textbf{I} follows from \cite{Fre} and that it is possible to extend $f$ for $t\geq T$ by solution of differential equation:
$$
g'(t)=a(g(t)), \ \ \ g(T)=f(T), \ t\geq T.
$$
Condition \textbf{II} is verified similarly to that in Theorem 4.1.
Condition \textbf{III} follows
from (\ref{4.13}), (\ref{4.14}) and Lemma 4.2.
\end{proof}

\section{Large Deviations for CEV model on half-line}

Consider $X_n(t)$, $t\geq 0$, that solves the following SDE (also known as the Constant Elasticity of Variance model, CEV).
$$
X_n(t)=1+\int_0^t\mu X_n(s)ds+\frac{1}{n^{1-\gamma}}\int_0^t \sigma (X_n(s))^\gamma dw(s),
$$
where
$\mu$ and $\sigma$ are arbitrary constants, $\gamma\in [1/2,1)$, $n>0$.
Existence and uniqueness of strong solution is given e.g. in \cite{Kle} and \cite{Delbaen}.

\begin{lemma} For any $N<\infty$ and $\varepsilon>0$ there exists $T=T_{N,\varepsilon}<\infty$
such that
\begin{equation}{\label{5.1}}
\lms_{n\to \infty}\frac{1}{n^{2(1-\gamma)}}\ln \mathbf{P}\biggl(
\sup\limits_{t\geq T}\frac{X_n(t)}{e^{\mu t } (1+t)^{\frac{1}{1-\gamma}}}>\varepsilon\biggl)\leq -N.
\end{equation}
\end{lemma}
\begin{proof}
Denote $\tau = \inf \{t\geq 0: X_n(t)=0\} \wedge \infty$.
Using It\^o's formula for $(X_n(t)e^{-\mu t })^{1-\gamma}$, $t\in [0,\tau)$, we have
\begin{eqnarray*}
(X_n(t)e^{-\mu t })^{1-\gamma}&=&1+\frac{1}{n^{1-\gamma}}\int_0^t \sigma (1-\gamma)e^{-\mu(1-\gamma)  s} dw(s)\\
&&
-\frac{1}{2n^{2(1-\gamma)}}\int_0^t\gamma(1-\gamma)\frac{\sigma^2e^{-\mu(1-\gamma)  s}}{(X_n(s))^{1-\gamma}}ds, \ \ \ t\in [0,\tau).
\end{eqnarray*}
Since$X_n(t)$ is non-negative with probability 1
\begin{eqnarray*}
(X_n(t)e^{-\mu t })^{1-\gamma}&\leq& 1+\frac{1}{n^{1-\gamma}}\int_0^t \frac{\sigma (1-\gamma)}{e^{\mu(1-\gamma)  s}} dw(s)\\
&\leq& 1+\frac{1}{n^{1-\gamma}} \biggl|\int_0^t \frac{\sigma (1-\gamma)}{e^{\mu(1-\gamma)  s}} dw(s) \biggl| \ \text{a.s.}, \ \ \ t\in [0,\tau).
\end{eqnarray*}
Since $X_n(t)\equiv 0$ for $t\geq \tau$, the above inequality trivially holds for $t\in [0,\infty)$.

 Therefore for $T>2/{\varepsilon^{1-\gamma}}$
\begin{eqnarray*}
 \mathbf{P}\biggl(
\sup\limits_{t\geq T}\frac{X_n(t)}{e^{\mu t } (1+t)^{\frac{1}{1-\gamma}}}>\varepsilon\biggl)&=&\mathbf{P}\biggl(\sup\limits_{t\geq T}\frac{(X_n(t)e^{-\mu t })^{1-\gamma}}{1+t}>\varepsilon^{1-\gamma}\biggl)\\
&\leq& \mathbf{P}\biggl(\sup\limits_{t\geq T}\biggl(\frac{1}{1+t}+\frac{1}{n^{1-\gamma}(1+t)} \biggl|\int_0^t \frac{\sigma (1-\gamma)}{e^{\mu(1-\gamma)  s}} dw(s) \biggl|\biggl)>\varepsilon^{1-\gamma}\biggl)\\
&\leq& \mathbf{P}\biggl(\sup\limits_{t\geq T}\frac{1}{n^{1-\gamma}(1+t)} \biggl|\int_0^t \frac{\sigma (1-\gamma)}{e^{\mu(1-\gamma)  s}} dw(s) \biggl|>\frac{\varepsilon^{1-\gamma}}{2}\biggl).
\end{eqnarray*}
Since
$|\sigma (1-\gamma)e^{-\mu(1-\gamma)s}| \leq \sigma (1-\gamma)$, using Lemma 4.1 we obtain (\ref{5.1}).
\end{proof}

Denote by  $\C^+$ the set of functions in $f\in \C$, such that $f(0)=1$, $f(t)\geq 0$ for all $t\geq 0$. Define a metric
in $\C^+$ by
$$
  \r^{\mu,\gamma}(f,g):=\sup_{t\ge 0}\frac{|f(t)-g(t)|}{e^{\mu t } (1+t)^{\frac{1}{1-\gamma}}}.
$$
It is obvious that the space $(\C^+,\r^{\mu,\gamma})$ is Polish.

\begin{theorem} The process $X_n$ satisfies LDP in space $(\C^+,\r^{\mu,\gamma})$ with rate $\frac{1}{n^{2(1-\gamma)}}$
and good rate function
$$ I(f)=\left\{
           \begin{array}{lcl}
                              \frac{1}{2}\int\limits_0^{\Theta(f)} \frac{(f'(t)-\mu f(t))^2}{(f(t))^{2\gamma}}dt, \text{ if}\ f(0)=1, \
                              f  \text{ is absolutely continuous}, \\
                              \infty, \ \text{otherwise,}\\
                              \end{array}
                              \right.$$
where $\Theta(f)=\inf\{t:~f(t)=0\}$.
\end{theorem}

\begin{proof}
Condition \textbf{I} follows from \cite{Kle} and that one can extend $f$ for $t\geq T$ by the solution of the differential equation
$$
g'(t)=\mu g(t), \ \ \ g(T)=f(T), \ t\geq T.
$$
We verify condition \textbf{II}.
Let $$\lms_{T\to \infty}I^T(f^{(T)})=\frac{r^2}{2}<\infty.$$
Write $f(t)$   as $f(t)=e^{\mu t}g(t)$, $g(0)=1$.
Then
$$
I^T(f^{(T)})=\frac{1}{2}\int_0^T e^{2\mu(1-\gamma) t}\biggl(\frac{g'(t)}{g^\gamma(t)}\biggl)^2dt.
$$
 Denote
\begin{equation}{\label{5.2}}
u^2(t)=e^{2\mu(1-\gamma) t}\biggl(\frac{g'(t)}{g^\gamma(t)}\biggl)^2, \ \ \ \lms_{T\to \infty}\int_0^T u^2(t)dt=r^2.
\end{equation}
Then
$$
\frac{g'(t)}{g^\gamma(t)}=e^{-\mu(1-\gamma) t}u(t), ~~~ g(0)=1.
$$
Solving, we have
$$
g^{1-\gamma}(t)=1+(1-\gamma)\int_0^te^{-\mu(1-\gamma) s}u(s)ds.
$$
Using Cauchy-Bunyakovskii inequality and (\ref{5.2})  we have
$$
|g^{1-\gamma}(t)|\leq 1+(1-\gamma) \biggl(\int_0^te^{-2\mu(1-\gamma) s}ds\int_0^tu^2(s)ds\biggl)^{1/2}\leq 1+r\biggl(\frac{1-\gamma}{2\mu}\biggl)^{1/2}.
$$
Hence
$$|f(t)|\leq e^{\mu t}\biggl(1+r\biggl(\frac{1-\gamma}{2\mu}\biggl)^{1/2}\biggl)^{\frac{1}{1-\gamma}}.$$
It now follows that
$$
\lms_{T\to \infty}\sup_{T\geq t}\frac{|f(t)|}{e^{\mu t}(1+t)^{\frac{1}{1-\gamma}}}=0.
$$
Condition \textbf{III} follows from Lemma 5.1. Theorem 5.1 is now proved.
\end{proof}

{\bf Acknowledgement.} This research was supported by the RFFI projects:
13--01--12415 ofi-m, 14--01--0020--a, and the Australian Research Council Grant DP120102728. The authors are grateful to the Referee for  comments and suggestions.

\end{document}